\documentclass[reqno,a4paper,12pt]{amsart}

\usepackage[T1]{fontenc}
\usepackage[english]{babel}
\usepackage[babel]{microtype}
\usepackage[a4paper, margin=1.1in]{geometry}

\usepackage{amsmath,amssymb,amsthm,amsrefs}

\usepackage{mathrsfs}

\usepackage{mathtools}
\usepackage{commath}

\usepackage{bbm}

\usepackage{cases}
\usepackage{enumitem}
\setlist[enumerate,1]{label=(\roman*)}

\usepackage{centernot}
\usepackage{booktabs}
\usepackage{multirow}
\usepackage{comment}
\usepackage{float}
\usepackage{scalerel}

\usepackage[pdftex, pdfborderstyle={/S/U/W 0}]{hyperref}
\hypersetup{
    colorlinks=true,
    linkcolor=blue!85!black,
    citecolor=blue!85!black,
    urlcolor=blue!85!black,
}

\usepackage{cleveref}

\usepackage{etoolbox}

\usepackage{subfiles}
\usepackage{pdfpages}

\numberwithin{equation}{section}

\theoremstyle{plain}

\newtheorem{theorem}    {Theorem}[section]
\newtheorem{lemma}      [theorem] {Lemma}
\newtheorem{conjecture} [theorem] {Conjecture}
\newtheorem{proposition}[theorem] {Proposition}
\newtheorem{corollary}  [theorem] {Corollary}

\theoremstyle{definition}

\newtheorem{claim}      [theorem] {Claim}

\newtheorem{observation}[theorem] {Observation}

\makeatletter
\newtheorem*{rep@theorem}{\rep@title}
\newcommand{\newreptheorem}[2]{%
\newenvironment{rep#1}[1]{%
 \def\rep@title{#2 \ref{##1}}%
 \begin{rep@theorem}}%
 {\end{rep@theorem}}}
\makeatother

\newreptheorem{theorem}{Theorem}
\newreptheorem{lemma}{Lemma}
\newreptheorem{claim}{Claim}
\newreptheorem{corollary}{Corollary}
\newreptheorem{definition}{Definition}

\newcommand{\defined}{\mathrel{\coloneqq}}

\DeclarePairedDelimiter{\p}{\lparen}{\rparen}

\renewcommand{\le}{\leqslant}
\renewcommand{\leq}{\leqslant}
\renewcommand{\ge}{\geqslant}
\renewcommand{\geq}{\geqslant}

\let\oldexists\exists
\let\exists\relax
\DeclareMathOperator{\exists}{\oldexists}
\let\oldforall\forall
\let\forall\relax
\DeclareMathOperator{\forall}{\oldforall}

\newcommand{\st}{\colon}
\undef{\set}
\DeclarePairedDelimiter{\set}{\lbrace}{\rbrace}

\DeclarePairedDelimiter{\card}{\lvert}{\rvert}

\undef{\mod}
\newcommand{\mod}[1]{\ (\mathrm{mod}\ #1)}

\undef{\abs}
\DeclarePairedDelimiterX{\abs}[1]
  {\lvert}{\rvert}{\ifblank{#1}{\,\cdot\,}{#1}}
\undef{\norm}
\DeclarePairedDelimiterX{\norm}[1]
  {\lVert}{\rVert}{\ifblank{#1}{\,\cdot\,}{#1}}
\DeclarePairedDelimiterX{\inner}[2]
  {\langle}{\rangle}{\ifblank{#1}{\,\cdot\,}{#1},\ifblank{#2}{\,\cdot\,}{#2}}

\DeclareMathDelimiter{\given}
  {\mathbin}{symbols}{"6A}{largesymbols}{"0C}
\DeclareMathOperator{\Prob}{\mathbb{P}}
\DeclarePairedDelimiterXPP{\prob}[1]
  {\Prob}{\lparen}{\rparen}{}
  {\renewcommand{\given}{\nonscript\;\delimsize\vert\nonscript\;\mathopen{}}
  \ifblank{#1}{\,\cdot\,}{#1}}
\DeclareMathOperator{\Expec}{\mathbb{E}}
\DeclarePairedDelimiterXPP{\expec}[1]
  {\Expec}{\lparen}{\rparen}{}
  {\renewcommand{\given}{\nonscript\;\delimsize\vert\nonscript\;\mathopen{}}
  \ifblank{#1}{\,\cdot\,}{#1}}
\DeclareMathOperator{\Var}{Var}
\DeclarePairedDelimiterXPP{\var}[1]
  {\Var}{\lparen}{\rparen}{}
  {\renewcommand{\given}{\nonscript\;\delimsize\vert\nonscript\;\mathopen{}}
  \ifblank{#1}{\,\cdot\,}{#1}}
\DeclareMathOperator{\Cov}{Cov}
\DeclarePairedDelimiterXPP{\cov}[2]
  {\Cov}{\lparen}{\rparen}{}{#1,#2}

\newcommand{\FF}{\mathbb{F}}

\newcommand{\NN}{\mathbb{N}}

\newcommand{\cA}{\mathcal{A}}
\newcommand{\cB}{\mathcal{B}}

\newcommand{\cF}{\mathcal{F}}

\newcommand{\cI}{\mathcal{I}}

\newcommand{\rec}{{\operatorname{rec}}}
\newcommand{\can}{{\operatorname{can}}}

\begin{document}

\title[The sandglass conjecture beyond cancellative pairs]
{The sandglass conjecture beyond cancellative pairs}

\author{Adva Mond \and Victor Souza \and Leo Versteegen}

\address{Department of Mathematics, King’s College London, Strand, London, WC2R 2LS, United Kingdom}
\email{adva.mond@kcl.ac.uk}

\address{Department of Mathematics, Cornell University, Ithaca NY 14853, USA; and Sidney-Sussex College, Cambridge, CB2 3HU, United Kingdom}
\email{vss39@cornell.edu, vss28@cam.ac.uk}

\address{Department of Mathematics, London School of Economics, London WC2A 2AE, United Kingdom}
\email{lversteegen.math@gmail.com}

\begin{abstract}
The sandglass conjecture, posed by Simonyi, states that if a pair $(\mathcal{A}, \mathcal{B})$ of families of subsets of $[n]$ is recovering then $|\mathcal{A}| |\mathcal{B}| \leq 2^n$.
We improve the best known upper bound to $|\mathcal{A}| |\mathcal{B}| \leq 2.2543^n$.
To do this we overcome a significant barrier by exponentially separating the upper bounds on recovering pairs from cancellative pairs, a related notion.
\end{abstract}

\maketitle


\section{Introduction}
In this paper we study the notion of recovering pairs, introduced by Ahlswede and Simonyi~\cite{Ahlswede1994-st}.
Denote by $2^X$ the family of all subsets of a set $X$.
Given two families $\cA, \cB \subseteq 2^X$, we say that $(\cA, \cB)$ is a \emph{recovering pair} over $X$ if for all $A, A' \in \cA$ and all $B, B' \in \cB$,
\begin{align}
\label{eq:left-recovering}
    A \setminus B &= A' \setminus B'  \implies A = A' \\
\label{eq:right-recovering}
    B \setminus A &= B' \setminus A' \implies B = B'.
\end{align}

In 1989, Simonyi made the following seminal conjecture about the maximal product size of recovering pairs, called the ``sandglass conjecture''. 
This conjecture\footnote{More precisely, they use ``sandglass conjecture'' to refer to a generalisation where the lattice $2^{[n]}$ is replaced by the product of finite chains. Still, Simonyi's original conjecture is generally referred to as the ``sandglass conjecture''.} was presented in print by Ahlswede and Simonyi~\cite{Ahlswede1994-st}.
\begin{conjecture}
\label{conj:sandglass}
Any recovering pair $(\cA,\cB)$ over $[n]$ satisfies $\card{\cA} \card{\cB} \leq 2^n$.
\end{conjecture}

If \Cref{conj:sandglass} holds, then it is sharp.
To see this, fix $S \subseteq [n]$, take $\cA$ to be the family of all subsets of $S$, and take $\cB$ to be the family of all subsets of $[n] \setminus S$.
Then $(\cA, \cB)$ is a recovering pair satisfying $\card{\cA} \card{\cB} = 2^{\card{S}} 2^{n - \card{S}} = 2^n$.

In this paper, we improve the best known upper bound for~\Cref{conj:sandglass}.
\begin{theorem}
\label{thm:recovering}
If $(\cA, \cB)$ is a recovering pair over $[n]$, then
    \begin{align*}
        \card{\cA} \card{\cB} \leq 2.2543^n.
    \end{align*}
\end{theorem}
This result is a corollary of our main result, \Cref{thm:main-sandglass}, where we overcome a significant obstacle in the study of the sandglass conjecture by showing an exponential separation between recovering pairs and cancellative pairs -- a related notion that was introduced by Holzman and K\"{o}rner~\cite{Holzman1995-xb} in their work on Simomyi's conjecture.

Given two families $\cA, \cB \subseteq 2^X$, we say that $(\cA, \cB)$ is a \emph{cancellative pair} over $X$ if for all $A, A' \in \cA$ and all $B, B' \in \cB$,
\begin{align}
\label{eq:left-cancellation}
    A \setminus B &= A' \setminus B \implies A = A' \\
\label{eq:right-cancellation}
    B \setminus A &= B' \setminus A \implies B = B'.
\end{align}
Note that any recovering pair is also a cancellative pair, and therefore any upper bound on cancellative pairs implies an upper bound for Simonyi's conjecture.
Holzman and K\"{o}rner~\cite{Holzman1995-xb} proved that for any cancellative pair $(\cA, \cB)$ over a ground set of size $n$, satisfies $|\cA||\cB| \le 2.3264^n$, implying the same upper bound for recovering pairs.

Following~\cite{Holzman1995-xb},
Soltész~\cite{Soltesz2018-sf} improved the upper bound to be $2.284^n$ for recovering pairs only.
More recently B. Janzer~\cite{Janzer2018-pb} improved the upper bound for cancellative pairs to $2.2682^n$ and hence also for recovering pairs, which was later improved by Nair and Yazdanpanah~\cite{nair2020and} to $2.2663^n$ for recovering pairs only, using techniques from coding theory.

In~\cite{Holzman1995-xb} the authors also observed that there exist\footnote{This is done by taking a product of the pair $(\cA, \cB)$ with itself, over the ground set $[3]$, where $\cA=\cB=\set{\set{1},\set{2},\set{3}}$.} cancellative pairs $(\cA,\cB)$ satisfying $\card{\cA} \card{\cB} \approx 2.08^n$ for infinitely many values of $n$, thus showing that the study of cancellative pairs is not enough to prove~\Cref{conj:sandglass}.
Later, Tolhuizen~\cite{Tolhuizen2000-ww} improved this lower bound by constructing cancellative pairs $(\cA, \cB)$ over $[n]$ with $\card{\cA} \card{\cB} = 2.25^{n-o(n)}$, for arbitrarily large $n$. It is an enticing open problem to determine if this construction is optimal for cancellative pairs.

Define the optimal rate\footnote{Using supermultiplicativity, one can show that both limits exist  (see~\Cref{sec:prelims}).} for recovering pairs as
\begin{align}
\label{eq:murec}
    \mu_\rec \defined \lim_{n \to \infty} \max \set{ \p[\big]{\card{\cA}\card{\cB} }^{1/n} \st (\cA,\cB) \text{ is a recovering pair over } [n]},
\end{align}
and, analogously for cancellative pairs, define
\begin{align}
\label{eq:mucan}
    \mu_\can \defined \lim_{n \to \infty} \max \set{ \p[\big]{\card{\cA}\card{\cB} }^{1/n} \st (\cA,\cB) \text{ is a cancellative pair over } [n] }.
\end{align}

It is clear that $2 \leq \mu_\rec \leq \mu_\can$.
The current state of art, due the works of Tolhuizen~\cite{Tolhuizen2000-ww}, B. Janzer~\cite{Janzer2018-pb}, and Nair and Yazdanpanah~\cite{nair2020and}, is that $2.25 \leq \mu_\can \leq 2.2682$, and $2 \le \mu_\rec \le 2.2663$.

The main result of this paper shows that $\mu_\rec$ and $\mu_\can$ are not equal, exponentially separating the maximal product size of recovering pairs from the one of cancellative pairs.

\begin{theorem}
\label{thm:main-sandglass}
\label{thm:RecUpr}
Let $(\cA, \cB)$ be a recovering pair over a ground set $X$ of size $n$, and let $\alpha = 0.27$ and $\theta = 2.222$.
Then we have
\begin{equation}
\label{eq:ABupr}
    \card{\cA} \card{\cB} \leq \max \set{2.2499, \theta^{\alpha} \cdot \mu_\can^{1-\alpha}}^n.
\end{equation}
In particular, $\mu_\rec < \mu_\can$.
\end{theorem}

\Cref{thm:recovering} follows immediately as a corollary from~\Cref{thm:main-sandglass} giving the new upper bound for the maximal product size of recovering pairs.

Given~\Cref{thm:main-sandglass}, any improvement of the upper bound on cancellative pairs to about $2.26^n$ would imply that $\mu_\rec < 2.25$.
Moreover, if we had $\mu_\can = 2.25$, then \Cref{thm:main-sandglass} would imply $\mu_\rec \leq 2.243$.

We remark that in the original paper~\cite{Ahlswede1994-st}, recovering pairs are actually defined in a slightly different way.
Instead of \eqref{eq:left-recovering} and \eqref{eq:right-recovering}, the required conditions are that $A \cup B = A' \cup B' \implies A = A'$ and that $A \cap B = A' \cap B' \implies B = B'$, for all $A, A' \in \cA$ and $B, B' \in \cB$.
It is easily shown that a pair $(\cA, \cB)$ is recovering in this sense if and only if the pair $(\set{A^c \st A \in \cA}, \cB )$ is recovering as defined by \eqref{eq:left-recovering} and \eqref{eq:right-recovering}.
Therefore, the two definitions are interchangeable for the objective of maximizing $\card{\cA} \card{\cB}$.
In this formulation, the sharp examples are given by taking $\cA$ to be all supersets of a fixed set $S \subseteq [n]$, and $\cB$ to be all its subsets, giving rise to the name ``the sandglass conjecture''.

In \cite{Ahlswede1994-st} the authors also discuss a one-sided variant of the sandglass problem, in which one aims to maximize $\card{\cA} \card{\cB}$ over all pairs satisfying \eqref{eq:left-cancellation} but not necessarily \eqref{eq:right-cancellation}.
We call such pairs \emph{left-cancellative}.
They observe an upper bound of $3^n$ and a lower bound of $6^{n/2}$ for this quantity. Using Tolhuizen's construction for large cancellative pairs, we determine the maximum size of a one-sided cancellative pair up to a sub-exponential error.

\begin{theorem}\label{thm:one-sided}
    The maximum of $\card{\cA} \card{\cB}$ over all left-cancellative pairs $(\cA,\cB)$ is $3^{n-o(n)}$.
\end{theorem}
The organisation of the remainder of the paper is as follows.
In~\Cref{sec:prelims}, we present preliminaries results we will use in the proof of~\Cref{thm:main-sandglass}.
In~\Cref{sec:restrictions}, we define an operation we perform to a recovering pair.
In~\Cref{sec:mainlem}, we prove the main lemma we need for the proof of~\Cref{thm:main-sandglass}, and in~\Cref{sec:beyond}, we prove~\Cref{thm:main-sandglass}.
Lastly, in~\Cref{sec:onesided}, we prove~\Cref{thm:one-sided}.

\section{Preliminaries}
\label{sec:prelims}

Given two recovering (cancellative) pairs $(\cA, \cB)$ and $(\cA', \cB')$ over disjoint ground sets $X$ and $X'$, respectively, we define their product to be the pair $(\cA'', \cB'')$ given by
\begin{align*}
    \cA'' \defined \set{A \cup A' \st A \in \cA,\; A' \in \cA'}, &&
    \cB'' \defined \set{B \cup B' \st B \in \cB, \; B' \in \cB'}.
\end{align*}
Note that the pair $(\cA'', \cB'')$ is also recovering (cancellative) over the ground set $X \cup X'$.
Moreover, we have $|\cA''||\cB''| = |\cA||\cB||\cA'||\cB'|$, implying that in~\eqref{eq:murec} and in~\eqref{eq:mucan} the limits exist.

We say that a pair of sets $(\cA,\cB)$ over a ground set of $n$ elements is \emph{$k$-uniform} if for every $A \in \cA$ and every $B \in \cB$, we have $\card{A} = \card{B} = k$.
We rely on the the following lemma (originally from~\cite{Soltesz2018-sf} and then reformulated in~\cite{Janzer2018-pb}), which shows, using a product argument, that in order to prove an upper bound for recovering (or cancellative) pairs in general, it is enough to consider $k$-uniform pairs only.
\begin{lemma}[Lemma 2 in~\cite{Janzer2018-pb}]
\label{lem:uniform}
Let $\mu>0$, and suppose that for all $n, k \in \NN$ with $k \leq n$, every $k$-uniform recovering pair over $n$-element ground set satisfies $\card{\cA} \card{\cB} \leq \mu^n$ for some constant $\mu$.
Then the same inequality holds for any (not necessarily uniform) recovering pair.
\end{lemma}

\subsection{Entropy of discrete random variables}
\label{sec:entropy}

We will make use of some basic facts about entropy, which we recall below.
See \cites{Cover2005-xs,Csiszar2011-lr} for standard references in information theory.
For a random variable $Z$ taking finitely many values $z_1, \dotsc, z_m$, for some $m \in \NN$, with probabilities $p(z_1), \dotsc, p(z_m)$, respectively, its entropy $H(Z)$ is defined as $H(Z) \defined \expec{-\log_2 p(Z)}$, or equivalently,
\begin{equation*}
  H(Z) \defined -\sum_{i \in [m]} p(z_i) \log_2 p(z_i).  
\end{equation*}

\begin{observation}
\label{obs:entropy}
    If $Z$ is a discrete random variable taking finitely many values, for some $m \in \NN$, with probabilities $p(z_1) \le \ldots \le p(z_m)$.
    Then we have $p(z_1) \ge 2^{-H(Z)}$.
\end{observation}
\begin{proof}
    If $p(z_1) < 2^{-H(Z)}$, this yields
    \[\log_2 p(z_1) < -H(Z) = \sum_{i=1}^{m} p(z_i) \log_2 p(z_i) \leq \log_2 p(z_1) \sum_{i=1}^{m} p(z_i) = \log_2 p(z_1),\]
    which is a contradiction.
\end{proof}

We also recall that the entropy function is \emph{subadditive} in the following sense.
If $Z_1$ and $Z_2$ are two discrete random variables, then their joint entropy (that is, the entropy of the joint random variable $(Z_1, Z_2)$) satisfy
\begin{equation*}
    H(Z_1, Z_2) \leq H(Z_1) + H(Z_2),
\end{equation*}
with equality only when $Z_1$ and $Z_2$ are independent.

If $Z$ attains only two values, one with probability $p \in (0,1)$ and the other with probability $1 - p$, then its entropy is given by the \emph{binary entropy function}
\begin{equation}
\label{eq:hp}
    h(p) \defined -p\log_2 p - (1 - p)\log_2(1 - p).
\end{equation}
We extend $h(p)$ to $p \in [0,1]$ continuously, setting $0 \log_2 0 \defined 0$.


\section{Filtered pairs}
\label{sec:restrictions}

In this section, we discuss some operations one can perform on a recovering or cancellative pair $(\cA, \cB)$ on $X$ to obtain new pairs $(\cA', \cB')$ on a subset of $X$.
This will allow us to perform inductive arguments with these families.

The first such operation is a simple \emph{restriction}, already considered in the work of Holzman and Körner~\cite{Holzman1995-xb}.
If $\cF$ is a family of subsets of $X$ and $i \in X$, define
\begin{align*}
    \cF_i \defined \set[\big]{F \in \cF \st i \in F }, \qquad
    \cF'_i \defined \set[\big]{F \in \cF \st i \notin F }.
\end{align*}
It is then clear that the following holds.

\begin{observation}
\label{obs:simple-restriction}
If $(\cA, \cB)$ is a recovering pair over $X$, then for every $i \in X$:
\begin{enumerate}
    \item $(\cA_i, \cB_i)$ is a recovering pair over $X$;
    \item $(\cA'_i, \cB'_i)$ is a  recovering pair over $X$ and over $X \setminus \set{i}$;
    \item $\p[\big]{ \set{A \setminus \set{i} \st A \in \cA_i}, \set{B \setminus \set{i} \st A \in \cB_i} }$ is a recovering pair over $X \setminus \set{i}$.
\end{enumerate}
Moreover, the same is true if we replace `recovering' by `cancellative'.
\end{observation}
We remark that \Cref{obs:simple-restriction} holds also for left-cancellative pairs and for left-recovering pairs (which are defined analogously to left-cancellative pairs).

For the remainder of the paper, we denote the relative size of each restriction for $i\in X$ by
\begin{align}
\label{eq:aibi}
    a_i \defined \frac{\card{\cA_i}}{\card{\cA}}, \quad
    a'_i \defined \frac{\card{\cA'_i}}{\card{\cA}}, \quad
    b_i \defined \frac{\card{\cB_i}}{\card{\cB}}, \quad
    b'_i \defined \frac{\card{\cB'_i}}{\card{\cB}}.
\end{align}
It follows then that for any $i \in X$, we have $a_i + a'_i = 1$ and $b_i + b'_i = 1$.
Furthermore, note that if the pair $(\cA, \cB)$ is $k$-uniform, for some integer $k$, then $\sum_{i \in X} a_1 = \sum_{i \in X} b_i = k$.
The proofs in~\cites{Holzman1995-xb,Janzer2018-pb} rely on an inductive argument based on this type of restriction and its properties, as given in~\Cref{obs:simple-restriction}.
In order to obtain a separation between the bounds $\mu_\can$ and $\mu_\rec$ however, we need to a different type of restriction, which exploits the recovering property.

We introduce a new type of restriction in order to differentiate recovering and cancellative pairs.
Let $(\cA, \cB)$ be a pair of families over $X$.
Then given $P \subseteq S \subseteq C \subset X$, the \emph{filtered} pair $(\cA_{C,S,P}, \cB_{C,S})$ of the pair $(\cA, \cB)$ is defined such that
\begin{align*}
    \cA_{C,S,P} &\defined \set{A \setminus C \st A \in \cA ,\; A \cap S = P}, \\
    \cB_{C,S} &\defined \set{B \setminus C \st B \in \cB ,\; C\setminus B = S}.
\end{align*}

We use three important statements regarding filtered pairs as described above.
Later we will only use filtered pairs where $C = A$ and $S = A \setminus B$, for some $A \in \cA$ and $B \in \cB$, but the properties described in the remainder of this section hold in larger generality.

\begin{observation}
\label{obs:filter-size}
Let $(\cA, \cB)$ be a cancellative pair over a set $X$.
Then for any sets $P \subseteq S \subseteq C \subset X$ such that there is $B \in \cB$ with $S = C \setminus B$, the filtered pair $(\cA_{C,S,P}, \cB_{C,S})$ satisfies
\begin{align*}
    \card{\cA_{C,S,P}}
    &= \card[\big]{\set{A \in \cA \st A \cap S = P }}, \\
    \card{\cB_{C,S}}
    &= \card[\big]{\set{B \in \cB \st C \setminus B = S}}.
\end{align*}
\end{observation}
\begin{proof}
For the first equality, we need to check that if two distinct $A_1, A_2 \in \cA$ satisfy $A_1 \cap S = P = A_2 \cap S$, then they give rise to distinct sets in $\cA_{C,S,P}$, meaning that $A_1 \setminus C \neq A_2 \setminus C$.
Suppose otherwise that $A_1 \setminus C = A_2 \setminus C$.
Let $B \in \cB$ have $C \setminus B = S$, or equivalently $C \setminus S = C \cap B$.
We have
\begin{align*}
    A_1  &= (A_1 \setminus C) 
    \sqcup (A_1 \cap (C \setminus S)) \sqcup (A_1 \cap S) = (A_1 \setminus C) \sqcup (A_1 \cap C \cap B) \sqcup P,
\end{align*}
and similarly, $A_2 = (A_2 \setminus C) \sqcup (A_2 \cap C \cap B) \sqcup P$.
Therefore,
\[A_1 \setminus B = \p[\big]{ (A_1 \setminus C) \setminus B } \sqcup P = \p[\big]{ (A_2 \setminus C) \setminus B } \sqcup P = A_2 \setminus B.\]
Since $(\cA, \cB)$ is a cancellative pair, we get $A_1 = A_2$.

For the second inequality, if $B_1, B_2 \in \cB$ are such that $B_1 \cap C = C \setminus S = B_2 \cap C$ and $B_1 \setminus C = B_2 \setminus C$, then
$B_1 = (B_1 \setminus C) \sqcup (B_1 \cap C) = (B_2 \setminus C) \sqcup (B_2 \cap C) = B_2$, completing the proof.
\end{proof}

\Cref{obs:filter-size} tells us that for every $A' \in \cA_{C,S,P}$, there is \emph{unique} set $A\in \cA$ such that $A'=A\setminus C$. We say that $A$ is the \emph{associated} set of $A'$. Similarly, the associated set for $B' \in \cB_{C,S}$ is the unique set $B \in \cB$ satisfying $B' = B\setminus C$.

Unfortunately, considering the filtered pair of a recovering pair may destroy the recovering property, but as we see bellow, it preserves the cancellation property.

\begin{observation}
\label{obs:filter-left-recovering}
Let $(\cA, \cB)$ be a recovering pair over a set $X$.
Then for any sets $P \subseteq S \subseteq C \subseteq X$, the filtered pair $(\cA_{C,S,P}, \cB_{C,S})$ is cancellative over $X \setminus C$.
\end{observation}
\begin{proof}
In fact, we show that for sets $A'_1, A'_2 \in \cA_{C,S,P}$ and $B'_1, B'_2 \in \cB_{C,S}$ the recovering property is preserved on one side, namely, they satisfy (\ref{eq:left-recovering}).
However, while they satisfy (\ref{eq:right-cancellation}), they need not satisfy (\ref{eq:right-recovering}), and thus the filtered pair is cancellative but not necessarily recovering.

Let $A'_1, A'_2 \in \cA_{C,S,P}$ and $B'_1, B'_2 \in \cB_{C,S}$ and let $A_1, A_2 \in \cA$ and $B_1, B_2 \in \cB$ be their associated sets.
We have
\begin{align*}
    A_1 \setminus B_1
    &= \p[\big]{(A_1\setminus C) \setminus (B_1 \setminus C)}
    \sqcup \p[\big]{(A_1 \cap C) \setminus (B_1 \cap C)} \\
    &= (A'_1 \setminus B'_1) \sqcup \p[\big]{(A_1 \cap C) \setminus (C \setminus S)} \\
    &= (A'_1 \setminus B'_1) \sqcup (A_1 \cap S) \\
    &= (A'_1 \setminus B'_1) \sqcup P.
\end{align*}
Similarly, we get $A_2 \setminus B_2 = (A'_2 \setminus B'_2) \sqcup P$.
Thus, if we assume that $A'_1 \setminus B'_1 = A'_2 \setminus B'_2$, we obtain
\begin{align*}
    A_1 \setminus B_1 = (A'_1 \setminus B'_1) \sqcup P = (A'_2 \setminus B'_2) \sqcup P
    = A_2 \setminus B_2.
\end{align*}
Since $(\cA, \cB)$ is recovering, we get that $A_1 = A_2$, and consequently $A'_1 = A'_2$.

Now let $B'_1, B'_2 \in \cB_{C,S}$ and $A' \in \cA_{C,S,P}$ and let $B_1, B_2 \in \cB$ and $A \in \cA$ be their associated sets.
We have
\begin{align*}
    B_1 \setminus A = \p[\big]{(B_1\setminus C) \setminus (A \setminus C)}
    \sqcup \p[\big]{(B_1 \cap C) \setminus (A \cap C)} = (B'_1 \setminus A') \sqcup \p[\big]{(C \setminus S) \setminus A },
\end{align*}
and similarly $B_2\setminus A = (B'_2 \setminus A') \sqcup \p[\big]{(C \setminus S) \setminus A }$.
Thus, if we assume that $B'_1 \setminus A' = B'_2 \setminus A'$, we obtain
\begin{align*}
    B_1 \setminus A = (B'_1 \setminus A') \sqcup \p[\big]{(C \setminus S) \setminus A } = (B'_2 \setminus A') \sqcup \p[\big]{(C \setminus S) \setminus A }
    = B_2 \setminus A.
\end{align*}
Since $(\cA, \cB)$ is recovering and in particular cancellative, we get that $B_1 = B_2$, and consequently $B'_1 = B'_2$.
\end{proof}

Finally, we now show that given $S \subseteq C \subseteq X$, there is always a choice of $P \subseteq S$ for which the family $\cA_{C,S,P}$ in the filtered pair is reasonably sized.
Recall the binary entropy function as defined in~\eqref{eq:hp}.

\begin{proposition}
\label{prop:AClarge}
Let $(\cA, \cB)$ be a recovering pair over $X$ and recall the definition of $a_i$ from~\eqref{eq:aibi}. Let $S \subseteq C \subseteq X$ be such that $S = C \setminus B$ for some $B \in \cB$.
Then there exists a subset $P \subseteq S$ such that
\begin{equation*}
    \log_2 \p[\Bigg]{ \frac{\card{\cA_{C,S,P}}}{\card{\cA}}} \geq -\sum_{i \in S} h(a_i).
\end{equation*}
\end{proposition}
\begin{proof}
Let $A \in \cA$ be chosen uniformly at random, and note that we have $a_i = \prob{i \in A}$ for all $i \in X$.
Let $X^S \defined (X_i)_{i \in S}$ be the characteristic function of $A \cap S$.
By the subadditivity of entropy, we have 
\begin{equation*}
    H(X^S) \leq \sum_{i \in S} H(X_i) = \sum_{i \in S} h(a_i).
\end{equation*}
Let $P \subseteq S$ be such that $\prob{A \cap S = P}$ is maximal.
By~\Cref{obs:entropy} we get that $\prob{A \cap S = P } \geq 2^{-H(X^S)}$.
By \Cref{obs:filter-size}, as recovering implies cancellative, and the above we get
\begin{equation*}
    \frac{\card{\cA_{C,S,P}}}{\card{\cA}}
    = \prob{A \cap S = P}
    \geq 2^{-H(X^S)}
    \geq 2^{-\sum_{i \in S} h(a_i)}. \qedhere
\end{equation*}
\end{proof}

\section{Uniform pairs and main lemma}
\label{sec:mainlem}

The following proposition is the main technical tool in our argument.
Roughly, it says that if a recovering pair $(\cA, \cB)$ does not have a relatively dense filtered pair, then it cannot be too large.
For the remainder of this section, we only consider filtered pairs such that $C = A$ and $S = A\setminus B$ for some sets $A \in \cA$ and $B \in \cB$.
\begin{proposition}
\label{prop:left-recovering-bound}
Let $(\cA, \cB)$ be a $k$-uniform recovering pair over $X$ and suppose that for every $A \in \cA$, $B \in \cB$ and $P \subseteq A \setminus B$ we have
\begin{align}
\label{eq:filtered-condition}
    \frac{\card{\cA_{A,A \setminus B,P}}\card{\cB_{A,A \setminus B}}}{\card{\cA}\card{\cB}} \leq \theta^{-k},
\end{align}
for some constant $\theta$.
Then we have
\begin{equation}
\label{eq:left-sum}
    \log_2 \card{\cA} \leq \sum_{i \in X} f(a_i, b_i, \theta),
\end{equation}
where $f(x,y,t) \defined x(1-y)h(x) + h \p[\big]{x(1-y)} - x \log_2 t$.
\end{proposition}
\begin{proof}
Let $Z_{\cA}$ and $Z_\cB$ be two independent random variables choosing uniformly at random sets from $\cA$ and from $\cB$, respectively.
Fix $A \in \cA$ and $B \in \cB$, then we have
\begin{equation}
\label{eq:prob}
    \prob{Z_{\cA} \setminus Z_{\cB} = A \setminus B}
    = \frac{\card{\set{(A',B') \in \cA \times \cB \st A' \setminus B' = A \setminus B } }}{\card{\cA} \card{\cB}}.
\end{equation}
Since the pair $(\cA, \cB)$ is recovering, we have
\begin{align}
\label{eq:recppty}
    \card{\set{(A',B') \in \cA \times \cB \st A' \setminus B' = A \setminus B } } = \card{\set{B' \in \cB \st A \setminus B' = A \setminus B } }.
\end{align}
Additionally, by \Cref{obs:filter-size} we have $\card{\set{B' \in \cB \st A \setminus B' = A \setminus B} } = \card{\cB_{A,A \setminus B}}$.
Therefore, \eqref{eq:prob} is equivalent to
\begin{equation}
\label{eq:size-of-A}
    \log_2 \card{\cA}
    = \log_2 \p[\Bigg]{ \frac{\card{\cB_{A,A\setminus B}}}{\card{\cB}} } - \log_2 \prob{Z_{\cA} \setminus Z_{\cB} = A \setminus B}.
\end{equation}

By \Cref{prop:AClarge}, there exists a subset $P \subseteq A \setminus B$ for which
\begin{equation*}
    \log_2 \p[\Bigg]{\frac{\card{\cA}}{\card{\cA_{A,A \setminus B,P}}} }
    \leq \sum_{i \in A \setminus B} h(a_i).
\end{equation*}
By rearranging condition~\eqref{eq:filtered-condition} and by the above we get
\begin{equation}
\label{eq:logB}
    \log_2\p[\Bigg]{ \frac{\card{\cB_{A,A \setminus B}}}{\card{\cB}} } \leq  \sum_{i \in A \setminus B} h(a_i) - k \log_2 \theta.
\end{equation}

Therefore, \eqref{eq:size-of-A} and \eqref{eq:logB} together imply
\begin{align*}
    \log_2 \card{\cA}
    &\leq \sum_{i \in A \setminus B} h(a_i) - \log_2 \prob{Z_{\cA} \setminus Z_{\cB} = A \setminus B} - k \log_2 \theta.
\end{align*}
Recall that $k$-uniformity implies that $k = \sum_{i \in X}a_i = \sum_{i \in X} b_i$, so we obtain
\begin{equation}
\label{eq:bound-random}
    \log_2 \card{\cA}
    \leq \sum_{i \in A \setminus B} h(a_i) - \log_2 \prob{Z_{\cA} \setminus Z_{\cB} = A \setminus B} - \sum_{i \in X} a_i \log_2 \theta.
\end{equation}

Define the function
\begin{equation*}
    F(A,B) \defined  \sum_{i \in A \setminus B} h(a_i) - \log_2 \prob{Z_{\cA} \setminus Z_{\cB} = A \setminus B}.
\end{equation*}
Choosing $A \in \cA$ and $B \in \cB$ independently and uniformly at random, we get
\begin{align*}
    \expec{F(A,B)} &= \sum_{i \in X} a_i(1 - b_i)h(a_i)  + H(Z_{\cA} \setminus Z_{\cB}) \\
    &\leq \sum_{i \in X} \p[\big]{ a_i(1 - b_i)h(a_i) + h\p[\big]{a_i(1 - b_i) } },
\end{align*}
where $H(Z_{\cA} \setminus Z_{\cB}) \leq \sum_{i \in X} h\p[\big]{a_i(1 - b_i)}$ follows from the subadditivity of entropy.
Considering $A_0 \in \cA$ and $B_0 \in \cB$ for which $F(A_0,B_0) \leq \expec{F(A,B)}$ in~\eqref{eq:bound-random} we get
\begin{align*}
    \log_2|\cA| \le \sum_{i \in X} \p[\big]{ a_i(1 - b_i)h(a_i) + h\p[\big]{a_i(1 - b_i) } - a_i\log_2 \theta} = f(a_i, b_i, \theta),
\end{align*}
implying the desired bound \eqref{eq:left-sum}.
\end{proof}
We remark that the crux of the proof of~\Cref{prop:left-recovering-bound} is captured in~\eqref{eq:recppty}.
Indeed, unlike other parts of the argument, \eqref{eq:recppty} holds due to the recovering property and does not hold for cancellative pairs in general.
This is what allows us to attain the desired separation in~\Cref{thm:main-sandglass}.

The proof of \Cref{prop:left-recovering-bound} used only that $A\setminus B$ determines $A$ but not that $B\setminus A$ determines $A$. Using the full power of the recovering property, we can symmetrise \Cref{prop:left-recovering-bound}.

\begin{corollary}
\label{cor:recovering-bound}
Let $(\cA, \cB)$ be a $k$-uniform recovering pair over $X$ and suppose that there exists a constant $\theta$ such that the following holds.
For every $A \in \cA$, $B \in \cB$ and for every $P_1 \subseteq A \setminus B$ and $P_2 \subseteq B\setminus A$ we have $\card{\cA_{A,A \setminus B, P_1}} \card{\cB_{A,A \setminus B}} \leq \theta^{-k} \card{\cA} \card{\cB}$, and $\card{\cB_{B,B \setminus A, P_2}} \card{\cA_{B,B \setminus A}} \leq \theta^{-k} \card{\cA} \card{\cB}$.
Then
\begin{equation*}
    \log_2 \card{\cA} \card{\cB}
    \leq \sum_{i \in X} g(a_i, b_i, \theta),
\end{equation*}
where $g(x,y,t) \defined f(x,y,t) + f(y,x,t)$, and $f$ is as in \Cref{prop:left-recovering-bound}.
\qed
\end{corollary}


\section{Beyond cancellation}
\label{sec:beyond}

We are now ready to prove \Cref{thm:RecUpr}.

\begin{proof}[Proof of \Cref{thm:RecUpr}]
We will show that \eqref{eq:ABupr} holds for every $k$-uniform recovering pair $(\cA, \cB)$, for any $1 \leq k \leq n$, and then the result will follow by \Cref{lem:uniform}.
Set $\alpha = 0.27$, $\theta = 2.222$, and $\mu = \max \set{2.2499, \theta^{\alpha} \cdot \mu_\can^{1-\alpha}}$.
We prove by induction that $\card{\cA}\card{\cB} \leq \mu^n$.

The base case $n = 1$ is trivially true, so fix $n \geq 2$.
Let $(\cA, \cB)$ be a $k$-uniform recovering pair over a ground set $X$ of size $n$ for some $1\le k \le n$, and for every $i \in X$, consider the restrictions $\cA_i$, $\cA'_i$, $\cB_i$ and $\cB'_i$, along with the parameters $a_i$, $a'_i$, $b_i$ and $b'_i$ as defined in \eqref{eq:aibi}.

Note that that $\theta < \mu_\can$, $0 \leq \alpha \leq 1/2$, $1 / (1 - 2\alpha) \leq \mu$ and $\theta^{\alpha} \mu_\can^{1 - \alpha} \leq \mu$.
If $k \leq \alpha n$, we proceed as in \cite{Janzer2018-pb}.
Indeed, since $(\cA, \cB)$ is $k$-uniform, we have
\begin{equation*}
    \frac{1}{n} \sum_{i \in X} a_i = \frac{1}{n} \sum_{i \in X} b_i = \frac{k}{n} \leq \alpha,
\end{equation*}
so there is $i \in X$ with $a_i + b_i \leq 2\alpha$, and hence with $a'_i + b'_i \geq 2 - 2\alpha$.
This implies that $a'_i b'_i \geq 1 - 2 \alpha \ge 0$, as $\alpha \leq 1/2$.
By \Cref{obs:simple-restriction}, we have that $(\cA'_i, \cB'_i)$ is a recovering pair over $X \setminus \set{i}$.
The induction hypothesis then gives that $\card{\cA'_i}\card{\cB'_i} \leq \mu^{n-1}$, thus
\begin{align*}
    \card{\cA} \card{\cB}
    &= \frac{\card{\cA_i'}\card{\cB_i'}}{a'_i b'_i}
    \leq \frac{1}{1 - 2\alpha} \mu^{n-1} \leq \mu^n,   
\end{align*}
since $1/(1 - 2\alpha) \leq \mu$, and we are done.

Henceforth, we assume $k \geq \alpha n$, and moreover, that $a'_i b'_i \leq 1 - 2 \alpha$ holds for all $i \in X$.
Assume further, that for every $A \in \cA$, $B \in \cB$ and for every $P_1 \subseteq A \setminus B$ and $P_2 \subseteq B\setminus A$, we have
\begin{align*}
    \frac{\card{\cA_{A,A \setminus B,P_1}}\card{\cB_{A,A \setminus B}}}{\card{\cA}\card{\cB}} \leq \theta^{-k} && \text{ and } && \frac{\card{\cB_{B,B \setminus A,P_2}}\card{\cA_{B,B \setminus A}}}{\card{\cA}\card{\cB}} \leq \theta^{-k}.
\end{align*}
Therefore, by \Cref{cor:recovering-bound} we get that
\begin{equation}
\label{eq:logAB}
    \frac{1}{n} \log_2 \card{\cA} \card{\cB}
    \leq \frac{1}{n}\sum_{i \in X} g(a_i, b_i, \theta),
\end{equation}
where
\[g(x,y,t) \defined f(x,y,t) + f(y,x,t) \]
and
\[f(x,y,t) \defined x(1 - y)h(x) + h\p[\big]{x(1 - y) } - x \log_2 t. \]

Recall that $(1 - a_i) (1 - b_i) = a'_i b'_i \leq 1 - 2 \alpha$ for all $i \in X$.
Given this restriction, we use the following claim to bound form above the value of the function $g(a_i,b_i,\theta)$.
As the proof of \Cref{cl:optimisation} is a straightforward but technical optimization argument, we include it as an appendix.
\begin{claim}
\label{cl:optimisation}
    For $x,y \in (0,1)$ and $\theta = 2.222$ we have
    \[g(x,y,\theta) \le \log_2(2.2499). \]
\end{claim}
By~\eqref{eq:logAB} together with \Cref{cl:optimisation} we get
\begin{align*}
    \card{\cA} \card{\cB} \le 2.2499^n,
\end{align*}
proving the statement.

Hence, we may assume, without loss of generality, that there are sets $A \in \cA$, $B \in \cB$ and $P \subseteq A \setminus B$ such that
\begin{align*}
    \frac{\card{\cA_{A,A \setminus B,P}}\card{\cB_{A,A \setminus B}}}{\card{\cA}\card{\cB}} \geq \theta^{-k}.
\end{align*}
\Cref{obs:filter-left-recovering} implies that the filtered pair $(\cA_{A,A \setminus B,P}, \cB_{A,A \setminus B})$ is cancellative over $X \setminus A$.
Since $(\cA, \cB)$ is $k$-uniform and $|X| = n$, we have $\card{X \setminus A} = n-k$, and thus, by the induction hypothesis,
\begin{equation*}
    \card{\cA_{A,A \setminus B,P}}\card{\cB_{A,A \setminus B}} \leq \mu_\can^{n-k}.
\end{equation*}
Therefore, we get
\begin{equation*}
    \card{\cA} \card{\cB}
    \leq \theta^{k} \card{\cA_{A,A \setminus B,P}}\card{\cB_{A,A \setminus B}}
    \leq \theta^k \mu_\can^{n-k}. 
\end{equation*}
Since $k \geq \alpha n$ and $\theta < \mu_\can$, we get
\begin{align*}
    \card{\cA}\card{\cB}
    \leq \theta^k \mu_\can^{n-k} \leq \theta^{\alpha n} \mu_\can^{n- \alpha n} \leq \mu^n,
\end{align*}
finishing the proof.
\end{proof}

\section{One-sided pairs}
\label{sec:onesided}
As stated in the introduction, in establishing the lower bound \Cref{thm:one-sided}, we mostly follow Tolhuizen's linear coding construction of a large cancellative pair \cite{Tolhuizen2000-ww}.
The proof of the upper bound was given in an argument by Ahlswede and Simonyi for one-sided recovering pairs from \cite{Ahlswede1994-st}, which we repeat here for the convenience of the reader.

\begin{proof}[Proof of~\Cref{thm:one-sided}]
    Let $(\cA, \cB)$ be a left-cancellative pair on $[n]$. For fixed $B\in \cB$, each set $A\in \cA$ has to have a different intersection with $[n]\setminus B$, whence $\card{\cA}\leq 2^{n-\card{B}}$. Thus, we obtain
    \begin{equation}\label{eq:one-sided-ub}
        \card{\cA}\card{\cB}=\sum_{i=0}^n \card{\cA}\card{\cB\cap [n]^{(i)}} \leq \sum_{i=0}^n 2^{n-i}\binom{n}{i}\leq 3^n.
    \end{equation}

    For the lower bound, we construct a left-cancellative pair $(\cA, \cB)$ with $|\cA||\cB| \ge 3^{n - o(n)}$. Given $n$, let $k\in [n]$ and consider a matrix $M\in \FF_2^{n\times (n-k)}$. We call $S\in [n]^{(n-k)}$ an \emph{information set} if the square submatrix of $M$ given by the rows indexed by $S$ is invertible. Denote the family of information sets of $M$ by $\cI(M)$. Associating vectors in $\FF_2^n$ with subsets of $[n]$ in the usual manner, we let $\cA_M$ be the image of $M$. Then, letting $\cB_M=\set{ [n] \setminus S \st S\in \cI(M)}$, the pair $(\cA_M,\cB_M)$ is left-cancellative. 
    
    Indeed, let $A, A'\in \cA_M$, $B\in \cB_M$ and let $v_A, v_{A'}\in \FF_2^{n-k}$ be vectors that give rise to $A$ and $A'$, respectively. If $A\setminus B=A'\setminus B$, then the 
    projections of $Mv_A$ and $Mv_{A'}$ to $\FF_2^{[n]\setminus B}$ are equal. However, these projections are the same as the images of $v_A$ and $v_{A'}$ under multiplication with the submatrix of $M$ given by the rows indexed by $[n]\setminus B$. Since this submatrix is invertible, the above can only happen when $v_A=v_{A'}$, meaning that $A=A'$.

    It thus remains only to find $k$ and $M$ so that $\card{\cA_M}\cdot\card{\cB_M}=2^{n-k}\card{\cI(M)}$ is at least $3^{n-o(n)}$. By considering the expected number of information sets for a uniformly random matrix, Tolhuizen showed (see Lemma 1 in \cite{Tolhuizen2000-ww}), that for all $k$ there exists $M$ such that
    \begin{equation*}
        \frac{\cI(M)}{\binom{n}{k}}\geq \prod_{i=1}^\infty (1-2^{-i})\approx 0.2888.
    \end{equation*}
    Thus, for every $k\in [n]$, there exists a left-cancelling pair of size at least
    \begin{equation*}
        \Omega\p[\big]{2^{n-k}\binom{n}{k}}.
    \end{equation*}
    Taking $k=\lfloor n/3\rfloor$ and applying Stirling's formula (or by recognizing that $2^{2n/3}\binom{n}{n/3}$ is the largest term of the binomial expansion in \eqref{eq:one-sided-ub}), we see that the above is $\Omega(3^n/\sqrt{n})=3^{n-o(n)}$, as desired.
\end{proof}

\paragraph{Acknowledgements.}

The first author was funded by UK Research and Innovation grant MR/W007320/2, and the second author was partially supported by ERC Starting Grant 101163189 and UKRI Future Leaders Fellowship MR/X023583/1.

\bibliography{sandglass-beyond-cancellation}


\appendix

\section{Proof of \texorpdfstring{\Cref{cl:optimisation}}{Claim 5.1}}
\label{sec:optiprob}

Recall that $f(x,y) \defined x(1-y)h(x) + h(x(1-y))$ and $g(x,y,t) \defined f(x,y) + f(y,x) - (x+y)\log_2 t$.
We restate the claim proved in this appendix.
\begin{repclaim}{cl:optimisation}
    For $x,y \in (0,1)$ and $\theta = 2.222$ we have
    \[g(x,y,\theta) \le \log_2(2.2499). \]
\end{repclaim}

\begin{proof}[Proof of \Cref{cl:optimisation}]
We show that $\theta$ satisfies conditions of a different optimisation problem.
Let $h^*(z)$ for $0<z<1$ be defined as follows. 
\begin{align*}
    h^*(z) = \begin{cases}
    h(z), &\text{ if } 0.01 < z < 0.99 \\
    h(0.01), &\text{ otherwise}
    \end{cases},
\end{align*}
and note that $h(z) \le h^*(z)$ for any $z \in [0,1]$.
Further define
\begin{align*}
    f^*(x,y) &\defined x(1-y)h^*(x) + h^*(x(1-y)) \\
    g^*(x,y,t) &\defined f^*(x,y) + f^*(y,x) - (x+y)\log_2 t,
\end{align*}
and note that, given $t$, we have $g(x,y,t) \le g^*(x,y,t)$ for any $x,y \in (0,1)$.
Hence, it is sufficient to show that for any $x,y \in (0,1)$ we have
\begin{align*}
    g^*(x,y,\theta) \le \log_2 (2.2499).
\end{align*}
    
It is routine to show that $g^*(x,y,\theta)$ is Lipschitz with a constant $25$ in either variable.
Hence, if we fix an integer $k$ and maximise $g^*(i/k,j/k,\theta)$ over $0 \leq i,j \leq k$, the maximum of $g^*(x,y,\theta)$ for all $0<x,y<1$ can be at most $25/k$ bigger.
Choosing $k=30 000$ and performing a computer check, we can hence bound $g^*(x,y,\theta)$ from above for all $0<x,y<1$ by $1.1687+\frac{25}{30 000}<1.1696$, while $\log_2 (2.2499)>1.1698$.
The statement follows.
\end{proof}

\end{document}